\documentclass[11pt,a4paper,english]{article}
\usepackage{amssymb}
\usepackage{amsfonts,yfonts}
\usepackage[all]{xy}
\usepackage{amsfonts,amscd,amssymb,amsthm}
\usepackage{amsfonts,amssymb,amsmath,amsxtra,amscd,amsthm,eucal,graphicx,graphics}
\usepackage{tikz}
\usepackage{enumerate}
 \usepackage[T1]{fontenc}
\usepackage[utf8]{inputenc}
\usepackage{babel}
\newcommand{\overbar}[1]{\mkern 1.5mu\overline{\mkern-1.5mu#1\mkern-1.5mu}\mkern 1.5mu}
\newcommand{\cupdot}{\mathbin{\mathaccent\cdot\cup}}

\theoremstyle{definition}
\newtheorem{theorem}{Theorem}[section]

\newtheorem{lemma}[theorem]{Lemma}
\newtheorem{corollary}[theorem]{Corollary}

\newtheorem{conjecture}[theorem]{Conjecture}
\newtheorem{question}[theorem]{Question}

 \title{A note on the real part of complex chromatic roots}
    \author{Jason Brown \and Aysel Erey\footnote{Corresponding author, e-mail: aysel.erey@gmail.com}}
    \date{Department of Mathematics and Statistics\\ Dalhousie University \\ Halifax, Nova Scotia, Canada B3H 3J5 \\[\baselineskip] }

\begin{document}

\maketitle

\begin{abstract}
A {\em chromatic root} is a root of the chromatic polynomial of a graph. While the real chromatic roots have been extensively studied and well understood, little is known about the {\em real parts} of chromatic roots. It is not difficult to see that  the largest real chromatic root of a graph with $n$ vertices is $n-1$, and indeed, it is known that the largest real chromatic root of a graph is at most the tree-width of the graph. Analogous to these facts, it was conjectured in \cite{dongbook} that the real parts of chromatic roots are also bounded above by both $n-1$ and the tree-width of the graph.

In this article we show that for all $k\geq 2$ there exist infinitely many graphs $G$ with tree-width $k$ such that $G$ has non-real chromatic roots $z$ with $\Re(z)>k$. We also discuss the weaker conjecture and prove it for graphs $G$ with $\chi(G)\geq n-3$.
\end{abstract}


 \section{Introduction}
Let $G$ be a simple graph of order $n$ and size $m$, and let $\chi(G)$ denote the chromatic number of $G$.  The {\em chromatic polynomial} $\pi(G,x)$ of $G$ counts the number of proper colourings of the vertices with $x$ colours. If $z\in \mathbb{C}$ satisfies $\pi(G,z)=0$, then $z$ is called a {\em chromatic root} of $G$ (the chromatic roots of graphs of order $8$ are shown in Figure~\ref{fig:order8}). A trivial observation is that all of $0,1,\ldots,\chi(G)-1$ are chromatic roots -- the chromatic number is merely the first positive integer that is {\em not} a chromatic root. The Four Colour Theorem is equivalent to the fact that $4$ is never a chromatic root of a planar graph, and interest in chromatic roots began precisely from this connection. The roots of chromatic polynomials have subsequently received a considerable amount of attention in the literature. Chromatic polynomials also have strong connections to the Potts model partition function studied in theoretical physics, and the complex roots play an important role in statistical mechanics (see, for example, \cite{sokal}).

\begin{figure}
\centering
\includegraphics[width=2.5in]{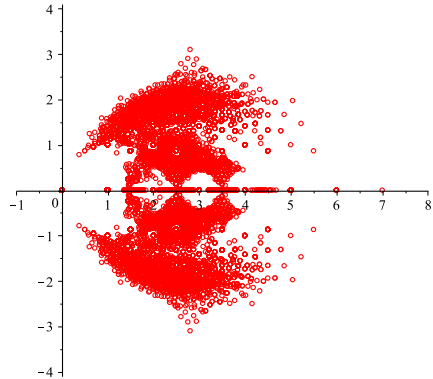}
\caption{Chromatic roots of graphs of order $8$.}
\label{fig:order8}
\end{figure}

A central problem has been to bound  the moduli of the chromatic roots in terms of graph parameters. There have been several results regarding this. Brown \cite{brownmonomial} showed that the chromatic roots of $G$ lie in $|z-1|\leq m-n+1$ and Sokal \cite{sokal} proved that the chromatic roots lie within $|z|\leq 7.963907\Delta(G)$, where $\Delta(G)$ is the maximum degree of the graph.

Another approach has been to study the {\em real} chromatic roots of graphs. It is not difficult to see that if $r$ is a real chromatic root of $G$ then $r\leq n-1$ with equality if and only if $G$ is a complete graph. In \cite{dong} it was proven that among all real chromatic roots of graphs with order $n \geq 9$, the largest non-integer real chromatic root is $\displaystyle{\frac{n-1-\sqrt{(n-3)(n-7)}}{2}}$, and  extremal graphs were determined. Moreover, Dong et.al \cite{dong koh 2, dong koh} showed that real chromatic roots are bounded above by $5.664\Delta(G)$ and max$\{\Delta(G), \lfloor n/3 \rfloor - 1\}$. 

The \textit{tree-width} of a graph $G$ is the minimum integer $k$ such that $G$ is a subgraph of a k-tree (given $q\in \mathbb{N}$, the class of $q${\em -trees} is defined recursively as follows: any complete graph $K_q$ is a $q$-tree, and any $q$-tree of order $n+1$ is a graph obtained from a $q$-tree $G$ of order $n$, where $n\geq q$, by adding a new vertex and joining it to each vertex of a $K_q$ in $G$). 
Thomassen \cite{thomassen}  proved that the real chromatic roots are bounded above by the tree-width of the graph. 

The problem of finding the largest {\em real part} of complex chromatic roots seems more difficult. In \cite{dongbook} the following conjectures on the real part of complex chromatic roots were proposed.

\begin{conjecture}\label{conjecture tree-width} \cite{dongbook}
Let $G$ be a graph with tree-width $k$. If $z$ is a root of $\pi(G,x)$ then $\Re(z)\leq k$.
\end{conjecture}

\begin{conjecture}\label{conjecture real n-1} \cite{dongbook}
Let $G$ be a graph of order $n$. If $z$ is a root of $\pi(G,x)$ then $\Re(z)\leq n-1$.
\end{conjecture}

It is clear that the Conjecture \ref{conjecture real n-1} is weaker than Conjecture \ref{conjecture tree-width}. In this work, first we present infinitely many counterexamples to Conjecture \ref{conjecture tree-width} for every $k\geq 2$ (Theorem \ref{tree-width thm}). Then, we consider Conjecture \ref{conjecture real n-1} and prove it for all graphs $G$ with $\chi(G)\geq n-3$ (Theorem \ref{n-3 realpart thm}). (Our numerical computations suggest that graphs which have large chromatic number are more likely to have  chromatic roots whose real parts are close to $n$.) 

\section{Main Results}

A polynomial $f(x)$ in $\mathbb{C}[x]$ is called  \textit{Hurwitz quasi-stable} or just {\textit quasi-stable} (resp. \textit{Hurwitz stable} or just {\textit stable})
if every $z\in \mathbb{C}$ such that $f(z) = 0$ satisfies $\Re(z) \leq 0$ (resp.  $\Re(z) < 0$). Observe that  $z$ is a root of $f(x)$ if and only if $z-c$ is a root of $f(x+c)$, so that every root $z$ of a  polynomial $f(x)$ satisfies $\Re(z) \leq c$ (resp. $\Re(z)<c$) if and only if the polynomial $f(x+c)$ is quasi-stable (resp. stable). Thus, bounding the real parts of roots of polynomials is closely related to the Hurwitz stability of polynomials. In the sequel, we will make use of this observation to prove both of our main results.

\subsection{Treewidth and the real part of complex chromatic roots}

It is not difficult to see that the tree-width of the complete bipartite graph $K_{p,q}$ is equal to $\operatorname{min}(p,q)$, and our counterexamples to Conjecture~\ref{conjecture tree-width} will be these graphs. Note that this conjecture clearly holds for $k=1$ since the tree-width of a graph is equal to $1$ if and only if the graph is a tree. Hence, our counterexamples are for $p\geq 2$.

We shall make use of a particular expansion of the chromatic polynomial. Let $G$ be a graph of order $n$ and size $m$. Suppose that $\beta: E(G) \rightarrow  \{1,2,\dots ,m \}$ is a bijection and $C$ a cycle in $G$. If $C$ has an edge $e$ such that $\beta(e)> \beta(e')$ for any $e'$ in $E(C)-\{e\}$ then  the path $C-e$ is called  a \textit{broken cycle} in $G$ with respect to $\beta$. Whitney's Broken-Cycle Theorem (see, for example, \cite{dongbook}) states that 
$$\pi(G,x)=\sum_{i=1}^{n}(-1)^{n-i}h_i(G)x^i, $$
where $h_i(G)$ is the number of spanning subgraphs of $G$ that have exactly $n-i$ edges and that contain no broken cycles with respect to $\beta$.

\begin{figure}

\centering
\includegraphics[scale=0.75]{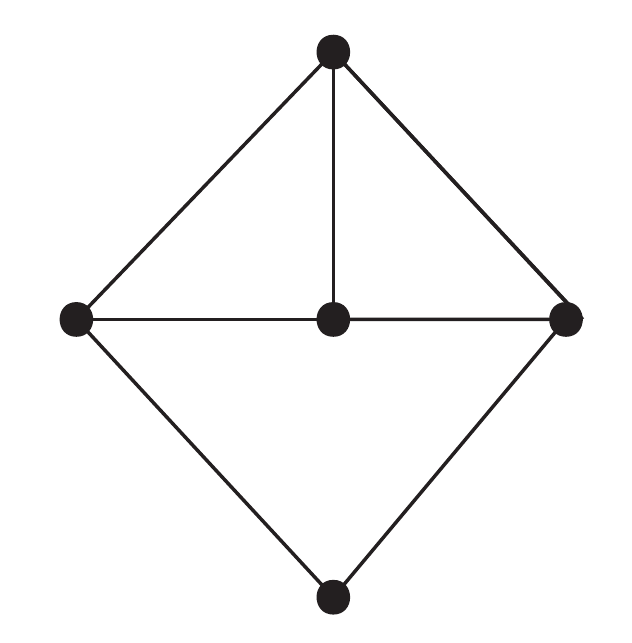}
\caption{The graph H}
\label{fig:treewidth}
\end{figure}

For two graphs $H$ and $G$, we denote by $\eta_G(H)$ (resp. $i_G(H)$) the number of subgraphs (respectively induced subgraphs) of $G$ which are isomorphic to $H$. For example, for the graph $H$ in Figure \ref{fig:treewidth}, we have $\eta_H(K_3)=i_H(K_3)=2$, $\eta_H(2K_2)=8$ and $i_H(2K_2)=0$. The following result gives formulas for the first few coefficients of the chromatic polynomial by counting certain (induced) subgraphs of the graph.

\begin{theorem}\label{chromaticcoeff}\cite[pg. 31-32]{dongbook} Let $G$ be a graph of order $n$ and size $m$, and let $g$ be the girth of the graph. Then 
$$\pi(G,x)=\sum_{i=1}^{n}(-1)^{n-i}h_i(G)x^i $$
is a polynomial in $x$ such that 
\begin{eqnarray*} 
h_{n-i} & = & {m \choose i} \mbox{ for } 0\leq i\leq g-2,\\
h_{n-g+1} & = & {m \choose g-1}-\eta_G(C_g),\\
h_{n-3} & = & {m \choose 3}-(m-2)\eta_G(K_3)-i_G(C_4)+2\eta_G(K_4), \mbox{ and}\\
h_{n-4} & = & {m \choose 4}-{m-2 \choose 2}\eta_G(K_3)+{\eta_G(K_3) \choose 2}-(m-3)i_G(C_4)-(2m-9)\eta_G(K_4)\\
 & & -i_G(C_5)+i_G(K_{2,3})+2i_G(H)+3i_G(W_5)-6\eta_G(K_5),
\end{eqnarray*}
where $H$ is the graph shown in Figure \ref{fig:treewidth} and $W_5$ is the wheel of order $5$.
\end{theorem}

\vspace{0.25in}
The first two items of Theorem \ref{chromaticcoeff} follow immediately from Whitney's Broken-Cycle Theorem and the expressions for $h_{n-3}$ and $h_{n-4}$ were obtained by Farrell in \cite{farrell}. A direct application of the previous result yields explicit formulas for the first few coefficients of the chromatic polynomials of complete bipartite graphs.

\begin{lemma}\label{lemmacompbip}Let  $p,q\geq 2$, $n=p+q$ and $\pi(K_{p,q},x)=\sum_{i=1}^{n}(-1)^{n-i}h_i(G)x^i, $ then
\begin{eqnarray*}
h_n & =& 1,\\
h_{n-1} & = & pq,\\
h_{n-2} & = & {pq \choose 2},\\
h_{n-3} & = & {pq \choose 3}-{q\choose 2}{p\choose 2}, \mbox{ and }\\
h_{n-4} & = & {pq \choose 4}-(pq-3){q\choose 2}{p\choose 2}+{q\choose 2}{p\choose 3}+{p\choose 2}{q\choose 3}. 
\end{eqnarray*}
\[\tag*{\qed}\]
\end{lemma} 

\vspace{0.25in}

A polynomial is called \textit{standard} if it is either identically zero or has positive leading coefficient, and is said to have \textit{only nonpositive roots} if it is either identically zero or has all of its roots real and nonpositive. Suppose that $f,g\in \mathbb{R}[x]$ both have only real zeros, that those of $f$ are $\zeta_1 \leq \dots \leq \zeta_a $ and that those of $g$ are $\theta_1 \leq \dots \theta_b$. We say that \textit{f interlaces g} if
$\mbox{deg }g = 1+\mbox{deg }f$ and the zeros of $f$ and $g$ satisfy $\theta_1 \leq \zeta_1 \leq \theta_2 \leq \dots \leq \zeta_a \leq \theta_{a+1}$.
We also say that \textit{f alternates left of g} if $\mbox{deg } f = \mbox{deg } g$ and the zeros of $f$ and
$g$ satisfy $\zeta_1 \leq \theta_1 \leq \zeta_2 \leq \dots \leq \zeta_a \leq \theta_a $. The notation $f \prec g$ stands for either
$f$ interlaces $g$ or $f$ alternates left of $g$. The following result which is known as Hermite-Biehler Theorem (see \cite{wagner}) characterizes  Hurwitz quasi-stable  polynomials via interlacing property.

\begin{theorem}[Hermite-Biehler Theorem]\label{hermite} Let $f(x)\in \mathbb{R}[x]$ be standard, and write $f(x)=f^e(x^2)+xf^o(x^2)$. Set $t=x^2$. Then $f(x)$ is Hurwitz quasi-stable if and only if both $f^e(t)$ and $f^o(t)$ are standard, have only nonpositive zeros, and $f^o(t)\prec f^e(t)$.
\end{theorem}

The \textit{Sturm sequence} of a real polynomial $f(t)$ of positive degree is a sequence of polynomials $f_0, f_1, f_2\dots$, where $f_0=f$, $f_1=f'$, and, for $i\geq 2$, $f_i=-\mbox{rem}(f_{i-1},f_{i-2})$, where $\mbox{rem}(h,g)$ is the remainder upon dividing $h$ by $g$. The sequence is terminated at the last nonzero $f_i$. The Sturm sequence of $f$ has \textit{gaps in degree} if there exist integers $j\leq k$ such that $\mbox{deg~}f_j<\mbox{deg~}f_{j-1}-1.$ Sturm's well known theorem (see, for example, \cite{brownrealpart}) is the following:

\begin{theorem}[Sturm's Theorem]\label{allrootsreal} Let $f(t)$ be a real polynomial whose degree and leading coefficient are positive. Then $f(t)$ has all real roots if and only if its Sturm sequence has no gaps in degree and no negative leading coefficients.
\end{theorem}

We are now ready to show that for complete bipartite graphs a chromatic root with real part greater than its tree width.

\begin{theorem}\label{tree-width thm}Suppose that $p\geq 2$ is fixed. Then, $\pi(K_{p,q})$ has a non-real root $z$ with $\Re(z)>p$ for all sufficiently large  $q$.
\end{theorem}

\begin{proof}Set $n=p+q$ and $\pi(K_{p,q},x)=\sum_{i=1}^{n}(-1)^{n-i}h_ix^i$. We will show that 
\[ \pi(K_{p,q},x+p)=\sum_{i=1}^{n}(-1)^{n-i}h_i(x+p)^i\]
is not Hurwitz quasi-stable. Rewriting $\pi(K_{p,q},x+p)=\sum_{i=1}^{n}a_ix^i$, we have
\begin{itemize}
\item $a_n=1$;
\item $a_{n-2}={n\choose 2}p^2-(n-1)ph_{n-1}+h_{n-2}$;
\item $a_{n-4}={n\choose 4}p^4-{n-1 \choose 3}p^{3}h_{n-1}+{n-2\choose 2}p^{2}h_{n-2}-(n-3)ph_{n-3}+h_{n-4}$.
\end{itemize}
Now we write $\pi(K_{p,q},x+p)=f^e(x^2)+xf^o(x^2)$. First, we suppose that $n$ is even and we look at the first three polynomials in the Sturm sequence $(f_{0},f_{1},f_{2},\ldots )$ of $f^e(t)$:
\begin{eqnarray*}
f_0 & = & t^{\frac{n}{2}}+a_{n-2}t^{\frac{n-2}{2}}+a_{n-4}t^{\frac{n-4}{2}}+\dots\\
f_1 & = & \frac{n}{2}t^{\frac{n-2}{2}}+a_{n-2}\frac{n-2}{2}t^{\frac{n-4}{2}}+a_{n-4}\frac{n-4}{2}t^{\frac{n-6}{2}}+\dots \\
f_2 & = & -\frac{2}{n^2}\Bigg(2na_{n-4}-(n-2)a_{n-2}^2\Bigg)t^{\frac{n-4}{2}}+\dots
\end{eqnarray*}

We can write $a_{n-4}$ and $a_{n-2}$ in terms of $p$ and $q$ by using Lemma \ref{lemmacompbip}, and then we can write $2na_{n-4}-(n-2)a_{n-2}^2$ as a quartic polynomial in $q$ where the coefficients are polynomial functions of $p$. More precisely, calculations show that $2na_{n-4}-(n-2)a_{n-2}^2$ is equal to

\begin{eqnarray*}
&&\left( \frac{1}{6}{p}^{2}-\frac{1}{6}p \right) {q}^{4}+ \left(\frac{1}{2}{p}^{4} -\frac{5}{3}{p}^{3}+{
\frac {11}{6}}\,{p}^{2}-\frac{2}{3} p \right) {q}^{3}+ \left( -\frac{5}{6} {p}^{5}+\frac{5}{3}{p}^{4}- \frac{5}{6} {p}^{3}-\frac{1}{3}{p}^{2}+\frac{1}{3} p \right) {q}^{
2}+ \\
&&\left( -\frac{1}{6}{p}^{8}+\frac{1}{3}{p}^{6}+\frac{1}{2}{p}^{5}-\frac{5}{6}{
p}^{4}-\frac{1}{6}{p}^{3}+\frac{1}{3}{p}^{2} \right) q + \left( -\frac{1}{6}{p}^{9}+\frac{1}{2}{p}^{8}-\frac{1}{3}{p}^{7} \right).
\end{eqnarray*}

Because $\frac{1}{6}p(p-1)> 0$  for fixed $p\geq 2$, it follows that the leading coefficient of $f_2$ is negative for all sufficiently large $q$.  Therefore, by Theorem \ref{allrootsreal}, we find that $f^e$ does not have all real roots and hence $\pi(K_{p,q},x+p)$ is not Hurwitz quasi-stable by Theorem \ref{hermite}. Thus, we obtain that
$\pi(K_{p,q},x)$ has a root $z$ with $Re(z)>p$ for all sufficiently large $q$ (that root cannot be a real number as we already noted that real chromatic roots are bounded by the tree-width of the graph). A similar argument works for $n$ odd but in this case one would work with the Sturm sequence of $f^o$ instead of $f^e$ (we leave the details to the reader).
\end{proof}

Since the tree-width of $K_{p,q}$ is equal to min$(p,q)$, the following corollary follows immediately.

\begin{corollary} For any integer $k\geq 2$, there exist infinitely many graphs which have tree-width $k$ and chromatic roots $z$ with $\Re(z)>k$.
\end{corollary}

\subsection{Bounding the real part of complex chromatic roots by $n-1$}

In this section we will use another form of the chromatic polynomial which known as the factorial form \cite{read}. The chromatic polynomial of $G$ is equal to
$$\sum_{i=\chi(G)}^n a_{i}(x)_{\downarrow i}$$ where $(x)_{\downarrow i} = x(x-1) \cdots (x-i+1)$ is the {\em falling factorial of $x$}. The following lemma gives an interpretation of the coefficient $a_{n-i}$ in terms of the number of certain subgraphs in the complement of the graph.

\begin{lemma}\cite{ereysigma, li} Let $\pi(G,x)=\sum a_i (x)_{\downarrow i} $, then $a_{n-i}$ counts the number of subgraphs of the form $\cupdot_{j=1}^kK_{{m_j}+1}$ in $\overbar{G}$ where $\sum_{j=1}^km_j=i$ and $m_j\in \mathbb{Z}^+$.
\end{lemma}

 From this lemma, we find that 
\begin{eqnarray*}
a_{n} & = & 1,\\
a_{n-1} & = & \eta_{\overbar{G}}(K_2) ~~ = ~~{n \choose 2} - |E(G)|,\\
a_{n-2} & = & \eta_{\overbar{G}}(K_3)+\eta_{\overbar{G}}(2K_2),\\
a_{n-3} & = & \eta_{\overbar{G}}(K_4)+\eta_{\overbar{G}}(K_3\cupdot \noindent K_2)+\eta_{\overbar{G}}(3K_2).
\end{eqnarray*}

We will need the following two results for the proof of Theorem \ref{n-3 realpart thm}. We ommit the proofs as the results are elementary.

\begin{lemma}\label{product-union}
Let $H$ and $K$ be two subgraphs of $G$, then
\[ \eta_G(H) \eta_G(K) \geq \eta_G(H\cupdot K).\tag*{\qed} \]
\end{lemma} 

\begin{lemma}\label{same order comparison}
Let $H_1,H_2, \dots , H_k$ be subgraphs of $G$ and  $r=\sum_{i=1}^k |V(H_i)|$. Then,
\[\eta_{G}(\cupdot_{i=1}^k H_i)\geq \eta_{G}(K_r).\tag*{\qed}\]
\end{lemma} 



\vspace{0.25in}

For our next result, we shall also need specific conditions for a low degree polynomial to be stable (see, for example, \cite[pg.181]{barbeau}).
 
\begin{theorem}[Stability tests for polynomials of degree $\leq 3$]\label{stability test}
The following are necessary and sufficient conditions for stability of polynomials of degree at most 3:
\begin{itemize}
\item A linear or quadratic polynomial is stable if and only if all the coefficients are of the same sign.
\item A cubic monic polynomial $f(x)=x^3+bx^2+cx+d$ is stable if and only if all its coefficients are positive and $bc>d$.
\end{itemize}
\end{theorem}

\vspace{0.25in}

We now conclude with showing that Conjecture~\ref{conjecture real n-1} holds for graphs with chromatic number at least $n-3$:

\begin{theorem}\label{n-3 realpart thm}
Let $G$ be a graph with $\chi(G)\geq n-3$. If $z$ is a root of $\pi(G,x)$ then $\Re(z)\leq n-1$ with equality if and only if $\chi(G)=n$.
\end{theorem}
\begin{proof} If $\chi(G)=n$ then $G=K_n$,  and if $\chi(G)=n-1$ then $G-u \cong K_{n-1}$ for some vertex $u$ of $G$. In both cases $G$ is chordal and hence $G$ has all integer roots, so the results follows as the largest integer roots is equal to $\chi(G)-1$. So we assume that $\chi(G)\geq n-2$.
We show that $\pi(G,x+n-1)$ is stable. First, we write
\begin{equation*}
\begin{split}
\pi(G,x+n-1)=f(G,x)  \prod_{i=1}^{\chi(G)}(x+n-i)       \\
\end{split}
\end{equation*}

and now it suffices to show that  $f(G,x)$ is  Hurwitz stable. Also, let $\pi(G,x)=\sum a_i(x)_{\downarrow i}$.

If $\chi(G)=n-2$, then 
\begin{equation*}
\begin{split}
f(G,x) & = a_{n-2}+a_{n-1}(x+1)+(x+1)x \\
& =  x^2+(1+a_{n-1})x+a_{n-1}+a_{n-2}        \\
\end{split}
\end{equation*}
Since all the coefficients are positive, the result is clear by Theorem \ref{stability test}.
Now, if $\chi(G)=n-3$, then
\begin{equation*}
\begin{split}
f(G,x) & = a_{n-3}+a_{n-2}(x+2)+a_{n-1}(x+2)(x+1)+(x+2)(x+1)x \\
& =  x^3+(3+a_{n-1})x^2+(2+3a_{n-1}+a_{n-2})x+2a_{n-1}+2a_{n-2}+a_{n-3}        \\
\end{split}
\end{equation*}
 Because $f(G,x)$ is a cubic polynomial with all coefficients positive, by Theorem \ref{stability test}, $f(G,x)$ is Hurwitz stable if and only if
$$(3+a_{n-1})(2+3a_{n-1}+a_{n-2}) > 2a_{n-1}+2a_{n-2}+a_{n-3}$$
which is equivalent to
\begin{equation}\label{stable cndtn for n-3}6+9a_{n-1}+a_{n-2}+3a_{n-1}^2+a_{n-1}a_{n-2} > a_{n-3}.\end{equation}
Observe that the left hand side of (\ref{stable cndtn for n-3}) is equal to
$$6+9\eta_{\overbar{G}}(K_2)+\eta_{\overbar{G}}(K_3)+\eta_{\overbar{G}}(2K_2)+3(\eta_{\overbar{G}}(K_2))^2+\eta_{\overbar{G}}(K_2)\cdot \eta_{\overbar{G}}(K_3)+\eta_{\overbar{G}}(K_2)\cdot \eta_{\overbar{G}}(2K_2).$$
By Lemma \ref{product-union}, $\eta_{\overbar{G}}(K_2)\eta_{\overbar{G}}(K_3)\geq \eta_{\overbar{G}}(K_3\cupdot K_2)$ and $\eta_{\overbar{G}}(K_2) \eta_{\overbar{G}}(2K_2)\geq \eta_{\overbar{G}}(3K_2)$. Also, by combining Lemma \ref{product-union} and Lemma \ref{same order comparison} we get $(\eta_{\overbar{G}}(K_2))^2\geq \eta_{\overbar{G}}(2K_2)\geq \eta_{\overbar{G}}(K_4)$.  Thus, the left hand side of (\ref{stable cndtn for n-3}) is strictly larger than
$$a_{n-3}=\eta_{\overbar{G}}(K_4)+\eta_{\overbar{G}}(K_3\cupdot K_2)+\eta_{\overbar{G}}(3K_2).$$ Hence, the inequality in (\ref{stable cndtn for n-3}) is established.
\end{proof}

\section{Concluding Remarks}
As we already mentioned, among all real chromatic roots of graphs with order $n \geq 9$, the largest non-integer real chromatic root is $\displaystyle{\frac{n-1-\sqrt{(n-3)(n-7)}}{2}}$. We pose the following question:

\begin{question}
Among all non-real chromatic roots of graphs with order $n$, what is the largest real part of a chromatic root of a graph of order $n$?
\end{question}

This problem seems more difficult and the answer must be at least $n-5/2$ (which is much bigger than the largest non-integer real root) as the graph $K_n-2K_2$ has non-real chromatic roots with real part equal to $n-5/2$. Indeed, we believe that this should be the true value.





Finally, we pose the following question:

\begin{question}
Let $G$ be a graph of order $n$. Is it true that if $z$ is a chromatic root of $G$ then $|z|\leq n-1$?
\end{question}

\vskip0.4in
\noindent {\bf \large Acknowledgments:} This research was partially supported by a grant from NSERC. 

\bibliographystyle{elsarticle-num}
\bibliography{<your-bib-database>}

\begin{thebibliography}{10}

\bibitem{barbeau}
E.J. Barbeau, \textit{Polynomials}, Springer-Verlag, New York (1989).

\bibitem{brownmonomial}
J.I. Brown,
\newblock Chromatic polynomials and order ideals of monomials, \textit{Discrete Math.} 189 (1998) 43-68.

\bibitem{ereysigma}
J.I. Brown, A. Erey,
On the roots of $\sigma$-polynomials, submitted.

\bibitem{brownrealpart}
J.I. Brown, C.A. Hickman,
On chromatic roots with negative real part, Ars Combin. 63 (2002) 211-221.

\bibitem{dong}
F.M. Dong,
\newblock The largest non-integer real zero of chromatic polynomials of graphs with fixed order, Discrete Math. 282 (2004), pp. 103-112.

\bibitem{dong koh 2}
F.M. Dong and K.M. Koh,
\newblock Two results on real zeros of chromatic polynomials, Combin. Probab. Comput. 13 (2004), pp. 809-813.

\bibitem{dong koh}
F.M. Dong and K.M. Koh,
\newblock Bounds For The Real Zeros of Chromatic Polynomials, Combinatorics, Probability and Computing 17 (2008), pp. 749-759.

\bibitem{dongbook}
Dong, F.M., Koh, K.M. and Teo, K.L., \textit{Chromatic Polynomials And
Chromaticity Of Graphs}, World Scientific, London, (2005).

\bibitem{farrell}
E. Farrell,
\newblock On chromatic coefficients, Discrete Math. 29(3) (1980), pp. 257-264.


\bibitem{li}
N.Z. Li,
On graphs having $\sigma$-polynomials of the same degree, Discrete Math. 110 (1992) 185-196.

\bibitem{read}
R.C. Read, An introduction to chromatic polynomials, J. Combin. Theory 4 (1968) 52-71.

\bibitem{sokal}
A.D. Sokal,
\newblock Bounds on the Complex Zeros of (Di)Chromatic Polynomials and Potts-Model Partition Functions, Combin. Probab. Comput. 10 (2001), 41-77.


\bibitem{thomassen}
C. Thomassen,
\newblock The zero-free intervals for chromatic polynomials of graphs, Combin. Prob. Comput. 6 (1997), pp. 497-506.

\bibitem{wagner}
D.G. Wagner,
\newblock Zeros of reliability polynomials and $f$-vectors of matroids, Combin. Prob. Comput. 9 (2000), pp. 167-190.


\bibitem{westbook}
D.B. West,
Introduction to Graph Theory, second ed., Prentice Hall, New York, 2001.



\end{thebibliography}

\end{document}